\documentclass[final,reqno,twoside,a4paper,11pt]{amsart}
\usepackage{mltools}
\usepackage{mlmath62,mlthm10-REAR,mllocal}
\usepackage{upref,amsmath,amssymb,amsthm,mathrsfs}
\usepackage{tikz,graphics,latexsym}
\usepackage{tikz-cd}
\usepackage{times}
\usepackage{amsmath}
\usepackage{amssymb}
\usepackage{amsthm}
\usepackage{verbatim}
\usepackage{amsmath}
\usepackage{amsfonts}
\usepackage{amssymb}
\usepackage[all]{xy}
\usepackage{xcolor}
\usepackage[utf8]{inputenc}
\usepackage{graphicx}
\usepackage{xfrac}

\usepackage[colorlinks=true,linkcolor=blue,citecolor=blue]{hyperref}

\usepackage[scaled]{helvet} 
\usepackage{courier} 
\usepackage[mathbf]{euler}
\usepackage{mllocal} 

\newcommand{\Sing}{\operatorname{Sing}}
\usepackage{pgfplots}

\numberwithin{equation}{section}

\begin{document}

\title[The Euler obstruction of a $1$-form on a determinantal singularity]{The Euler obstruction of a $1$-form on a determinantal singularity}

\author{Anne Frühbis-Krüger}
\address{Institut für Mathematik, Carl-von-Ossietzky Universität, Germany}
\email{anne.fruehbis-krueger@uol.de}

\author{Hellen Santana}
\address{Departamento de Matem\'atica, 
	Universidade Federal de S\~ao Carlos (UFSCar),
	Brazil}
\email{hellen@ufscar.br}

\thanks{Anne Frühbis-Krüger is partially supported by DFG under TRR 195 "Symbolic Tools in Mathematics and their application'', Hellen Santana is partially supported by CAPES-PROBAL under grant $88887.180502/2025-00$.}



\begin{abstract}
In this work, we investigate the connections between the local Euler obstruction and the Poincaré-Hopf-Nash (PHN) index of a $1$-form in the setting of determinantal singularities. As an application, we provide explicit computations of the Euler obstruction of a function with a stratified isolated singularity at the origin defined on an IDS with rigid singularities.
\end{abstract}

\maketitle

\tableofcontents

\section{Introduction}

The study of vector fields and $1$-forms on singular varieties is both classical and highly active in singularity theory, often revealing crucial information about the underlying variety. For a smooth complex manifold, the Poincaré-Hopf theorem relates the sum of the indices of a vector field (or a $1$-form) with isolated zeros to the Euler characteristic of the manifold \cite{milnor1997topology}. However, in the presence of singularities of the underlying variety, the classical notion of the index is no longer well-defined, and several generalizations have been proposed to capture the local topological and geometric data \cite{bonatti1994index,brasselet2010euler,BLS,Brasselet2009,dutertre2010radial,Ebeling2023,ruas2014codimension,nuno2013vanishing,nuno2018erratum,seade2007indices}.

Among the most fundamental generalizations are the \textit{radial index} and the \textit{local Euler obstruction of a 1-form}, extensively studied by Ebeling and Gusein-Zade (see \cite{ebeling2005radial,ebeling2015equivariant,chernobstructions,gusein2018index}). These invariants are constructed using the Whitney stratification of the singular space and the Nash transformation, providing deep insights into the behavior of a $1$-form near an isolated singular point on the variety. 

While these concepts apply to general analytic varieties, focusing on specific classes of singularities allows for a more detailed study, often revealing richer combinatorial and geometric structures. Beyond hypersurfaces and complete intersections, a prominent class is that of determinantal singularities, which are defined by the loci where the rank of a matrix with holomorphic entry functions drops below a given threshold. Within this context, Essentially Isolated Determinantal Singularities (EIDS) form a particularly well-behaved subclass, admitting a canonical deformation known as an \textit{essential smoothing}. However, unlike the case of isolated complete intersection singularities (ICIS), the essential smoothing of an EIDS is not necessarily a smooth manifold, making the study of $1$-forms on these spaces intrinsically more complex.

To address the problem of indices of $1$-forms on determinantal singularities, the \textit{Poincaré-Hopf-Nash index} (PHN-index) was introduced in \cite{eg-determinantal}. This index is defined as an obstruction to extending the dual Nash bundle in a particular setting, see Definition \ref{PHN-index definition}.

In this work, we review local Euler obstruction and the Poincaré-Hopf-Nash index of a $1$-form and explore connections between these invariants in the setting of determinantal singularities. Figure \ref{figure} illustrates this connection in a simple example: considering a $1$-form over the cone $X,$ the Euler obstruction of $\omega$ at the origin, the only singularity of $X,$ can be determined as the number of singular points of $\omega$ on $X\setminus\{0\}$ (blue dots in the picture). Passing to an essential smoothing of $X,$ the original singular points of $\omega$ on the non-singular part of $X$ persist. However, the singular point may give rise to new singular points (shown as red points) of $\omega|_{\tilde{X}}$. The number of such newly created singular points may depend on the smoothing. 

\begin{center}
\begin{figure}[htbp]\label{figure}
    \centering
    \begin{minipage}{0.43\textwidth}
        \centering
        \includegraphics[width=\linewidth]{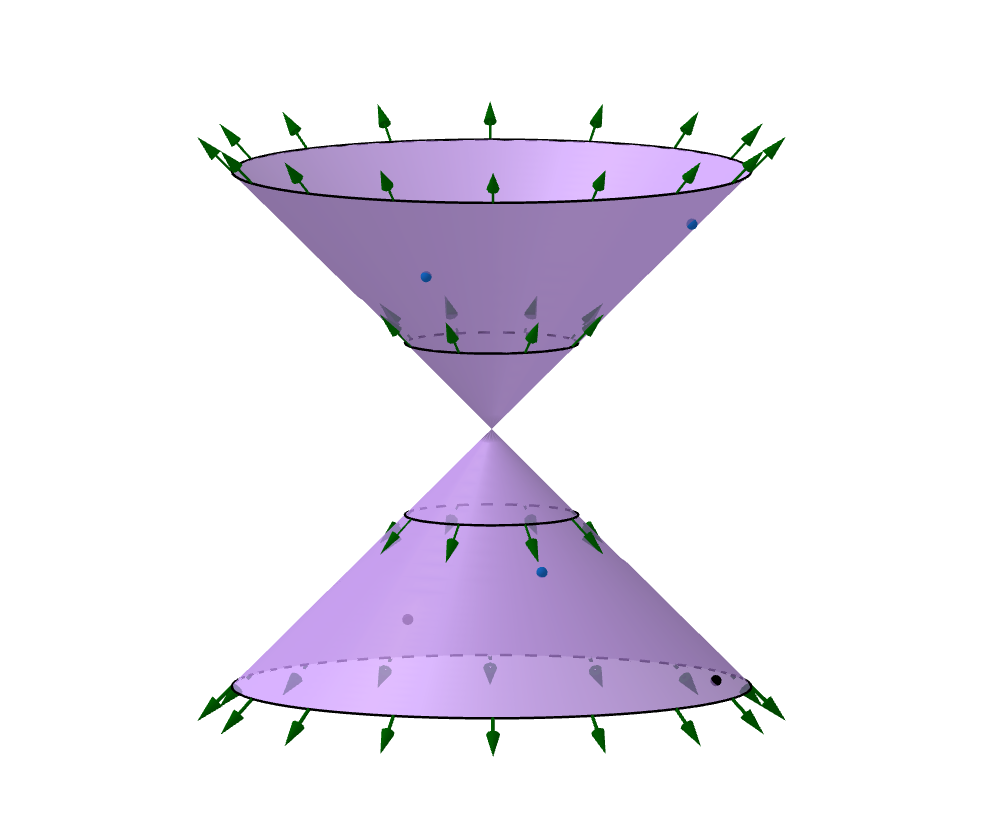}
    \end{minipage}
    \hfill 
    \begin{minipage}{0.56\textwidth}
        \centering
        \includegraphics[width=\linewidth]{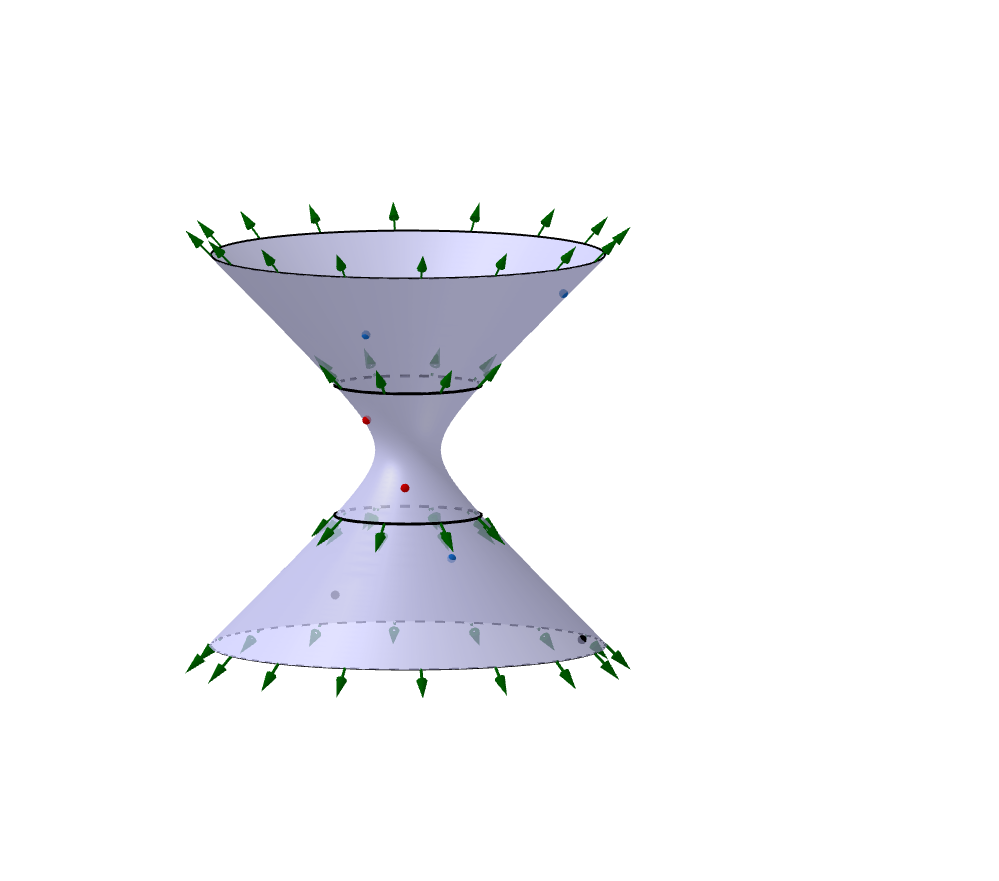}
    \end{minipage}
    
    \caption{$1$-form over the cone $X$ (left) and $1$-form over an essential smoothing $\tilde{X}$ of $X$ (right).}
    \label{fig:comparison_cones}
\end{figure}
\end{center}

Our main result shows that the Euler obstruction of $\omega$ at the origin equals the difference between the PHN-index of $\omega$ and  the PHN-index of a generic linear form $dl.$ This can be seen as adjusting the PHN-index of $\omega$ by measuring the contribution of the smoothing as the PHN-index of $dl.$ The paper is organized as follows. In Section 2, we recall the definitions and main properties of the radial index and the local Euler obstruction of a $1$-form on general singular varieties. We then specialize to the framework of determinantal singularities and essential smoothings. In Section 3, we relate the local Euler obstruction of a $1$-form with the Poincaré-Hopf-Nash index of the $1$-form defined on an EIDS, providing examples in which we compute these numbers. As an application, we compute the Euler obstruction of a function with a stratified isolated singularity at the origin defined on an EIDS. In Section 4, we analyze the developed formulas in the context of a smoothable IDS.

\vspace{0.5cm}

\section{Basic notions and known results}

Before stating and proving our result, we recall the radial and Poincaré-Hopf-Nash indices and the Euler obstruction of a $1$-form and their properties. 

\subsection{The Radial index of a 1-form}

We refer to \cite{ebeling2005radial} fo the complete description of the radial index.  Let $(X, 0) \subset (\mathbb{R}^N, 0)$ be a germ of a real analytic variety, and let $\omega$ be a (continuous) $1$-form on a neighborhood of the origin in $\mathbb{R}^N$.

\vspace{0.3cm}
\begin{definition}
    The $1$-form $\omega$ is \textit{radial} on $(X, 0)$ if, for an arbitrary nontrivial analytic arc $\varphi : (\mathbb{R}, 0) \to (X, 0)$ on $(X, 0)$, the value of the $1$-form on the tangent vector $\dot{\varphi}(t)$ is positive for sufficiently small positive $t$.
\end{definition} 
\vspace{0.3cm}

Let $X = \bigcup_{i=0}^q V_i$ be a Whitney stratification of a germ $(X, 0)$, with $V_0 = \{0\}$. Let $\omega$ be an (arbitrary continuous) $1$-form on a neighborhood of the origin in $\mathbb{R}^N$ with an isolated singular point \footnote{A singular point of a $1$-form is a zero of that $1$-form.} on $(X, 0)$ at the origin. Let $\epsilon > 0$ be small enough so that in the closed ball $B_\epsilon$ of radius $\epsilon$ centered at the origin in $\mathbb{R}^N$, the $1$-form $\omega$ has no singular points on $X \setminus \{0\}$. There exists a $1$-form $\tilde{\omega}$ on $\mathbb{R}^N$ such that:
\begin{enumerate}
    \item[(1)] the $1$-form $\tilde{\omega}$ coincides with the $1$-form $\omega$ on a neighborhood of the sphere $S_\epsilon = \partial B_\epsilon$;
    \item[(2)] the $1$-form $\tilde{\omega}$ is radial on $(X, 0)$ at the origin;
    \item[(3)] in a neighborhood of each singular point $x_0 \in (X \cap B_\epsilon) \setminus \{0\}$, where $x_0 \in V_i$ and $\dim V_i = k$, there exists a (local) analytic diffeomorphism $h : (\mathbb{R}^N, \mathbb{R}^k, 0) \to (\mathbb{R}^N, V_i, x_0)$ such that $h^*\tilde{\omega} = \pi_1^*\tilde{\omega}_1 + \pi_2^*\tilde{\omega}_2$, where $\tilde{\omega}_1$ is a germ of a $1$-form on $(\mathbb{R}^k, 0)$ with an isolated singular point at the origin, $\tilde{\omega}_2$ is a radial $1$-form on $(\mathbb{R}^{N-k}, 0)$, and $\pi_1$ and $\pi_2$ are the natural projections $\pi_1 : \mathbb{R}^N \to \mathbb{R}^k$ and $\pi_2 : \mathbb{R}^N \to \mathbb{R}^{N-k}$, respectively.
\end{enumerate}

\begin{definition}
    The radial index $\mathrm{ind}_{rad}(\omega, X,0)$ of the $1$-form $\omega$ on the variety $X$ at the origin is the sum
\[
\mathrm{ind}_{rad}(\omega, X,0) = 1 + \sum_{i=1}^q \sum_{Q \in \Sing \tilde{\omega} \cap V_i} \mathrm{ind}_{rad}(\tilde{\omega},V_i,Q),
\]
where the sum is taken over all singular points of the $1$-form $\tilde{\omega}$ on $(X \setminus \{0\}) \cap B_\epsilon$.
\end{definition} 

Now let $(X, 0) \subset (\mathbb{C}^N, 0)$ be a germ of a complex analytic variety of pure dimension $n$ and let $\omega$ be a complex continuous $1$-form on a neighborhood of the origin in $\mathbb{C}^N$. There is a one-to-one correspondence between complex $1$-forms on a complex manifold $M^n$ and real $1$-forms on it, considered as a real $2n$-dimensional manifold (see \cite{ebeling2005radial}), and this explains the following definition.

\begin{definition}
   Let $Re(\omega)$ be the real component of $\omega$ on $X.$ The (\textit{complex radial}) \textit{index} $\text{ind}_{X,0}^{\mathbb{C}} \omega$ of the complex $1$-form $\omega$ on $X$ at the origin is defined by
\[
\mathrm{ind}_{rad}^{\mathbb{C}}(\omega;X,0) = (-1)^n \mathrm{ind}_{rad}(Re(\omega);X,0).
\]
\end{definition}

From now on, by radial index we always mean complex radial index, and we will omit the upper index $\mathbb{C}.$

\subsection{The local Euler obstruction of a 1-form}\label{Euler onstruction of a 1-Form}

In \cite{ebeling2005radial}, the authors introduced the notion of the Euler obstruction $Eu_{X,0}\omega$ of a 1-form.

Let $(X, 0) \subset (\mathbb{C}^N, 0)$ be a germ of a complex analytic variety with a Whitney stratification $X = \bigcup_{i=0}^q V_i$, with $V_0 = \{0\}$, and let $\omega$ be a $1$-form on a neighborhood of the origin in $\mathbb{C}^N$ with an isolated singular point on $X$ at the origin. Let $\epsilon > 0$ be small enough such that the $1$-form $\omega$ has no singular points on $X \setminus \{0\}$ inside the ball $B_\epsilon$. Let $\nu : \hat{X} \to X$ be the Nash transformation of the variety $X$ and $\hat{T}$ be the Nash bundle over $\hat{X}$, given by the pullback of the tautological bundle on the Grassmann manifold $G(n, N)$.  The $1$-form $\omega$ gives rise to a section $\hat{\omega}$ of the dual Nash bundle $\hat{T}^*$ over the Nash transform $\hat{X}$ without singular points outside the preimage of the origin.

\begin{definition}\label{Eu definition}
The \textit{Euler obstruction} $Eu_{X,0} \omega$ of the $1$-form $\omega$ on $X$ at the origin is the obstruction to extending the nonzero section $\hat{\omega}$ from the preimage of a neighborhood of the sphere $S_\epsilon = \partial B_\epsilon$ to the preimage of its interior. More precisely, it is the evaluation of the obstruction cocycle in $H^{2n}(\nu^{-1}(X \cap B_\epsilon), \nu^{-1}(X \cap S_\epsilon))$ on the fundamental class of the pair $(\nu^{-1}(X \cap B_\epsilon), \nu^{-1}(X \cap S_\epsilon))$.
\end{definition}

For two strata $V_i$ and $V_j$ such that $V_i \subset \overline{V_j}$ and $V_i \neq V_j$, let $N_{ij}$ be the normal slice of the variety $V_j$ to the stratum $V_i$ at a point of it, and let $n_{ij}$ be the index of a generic (nonvanishing) $1$-form $dl$ on $N_{ij}$: $n_{ij} = (-1)^{\dim N_{ij} - 1} \bar{\chi}(M_{l|_{N_{ij}}})$, where $M_{l|_{N_{ij}}}$ is the Milnor fibre of a linear function $l$ on $N_{ij}$ and $n_{ii} = 1$. Let $m_{ij}$ be the (Möbius) inverse of the function $n_{ij}$ on the partially ordered set of strata (see \cite{ebeling2005radial, hall1998combinatorial}), i.e.
\[
\sum_{i \le j \le k} n_{ij} m_{jk} = \delta_{ik}
\]

\begin{proposition}\cite[Corollary 1]{ebeling2005radial}\label{Euradial}
    One has $Eu_{X,0} \omega =\displaystyle \sum_{i=0}^q m_{iq} \cdot \mathrm{ind}_{rad}(\omega,\overline{V_i},0)$.
\end{proposition}

\subsection{Determinantal singularities}

Let $(M_{m,n}^t, 0) \subset (\text{Mat}(m, n; \mathbb{C}), 0)$ be the generic determinantal variety of type $(m, n, t)$:
\[
M_{m,n}^t = \{M \in \text{Mat}(m, n; \mathbb{C}) : \text{rank } M < t\}.
\]
The canonical rank stratification by
\[
S_{m,n}^s = M_{m,n}^s \setminus M_{m,n}^{s-1}
\]
for $0 < s \le \min\{m, n\} + 1$ is a Whitney stratification of $\text{Mat}(m, n; \mathbb{C})$ and $M_{m,n}^t$.

\begin{definition}(\cite{eg-determinantal}). A determinantal singularity of type $(m, n, t)$ is given by a holomorphic map germ
\[
F: (\mathbb{C}^N, 0) \to (\text{Mat}(m, n; \mathbb{C}), 0)
\]
such that the space
\[
(X, 0) := (F^{-1}(M_{m,n}^t), 0) \subset (\mathbb{C}^N, 0)
\]
has \textit{expected codimension} $\text{codim}(X, 0) = \text{codim } M_{m,n}^t = (m - t + 1)(n - t + 1)$.
    
\end{definition} 

A determinantal singularity $(X, 0)$ given by a matrix $F$ is called \textit{essentially isolated} if the map $F$ is transverse to the rank stratification of $\text{Mat}(m, n; \mathbb{C})$ in a punctured neighborhood of the origin.

\begin{definition}\label{essential smoothing} An \textit{essential smoothing} of $(X,0)$ is a flat family
\[
\begin{array}{ccc} 
X & \longhookrightarrow & \mathcal{X} \\ 
\downarrow & & \downarrow u \\ 
\{0\} & \longrightarrow & \mathbb{C} 
\end{array}
\]
coming from a stabilization
\[
\tilde{F} : (\mathbb{C}^N, 0) \times (\mathbb{C}, 0) \to (\operatorname{Mat}(m, n; \mathbb{C}), 0) \times (\mathbb{C}, 0)
\]
of the map $F$. That is, $\tilde{F} = \tilde{F}(x, u) = (\tilde{F}_u(x), u)$ with $\tilde{F}_0 = F$ and $\tilde{F}_u$ transversal to $M_{m,n}^t$ for all sufficiently small $u \neq 0$. Then, the total space of the family above appears as $\mathcal{X} = \tilde{F}^{-1}(M_{m,n}^t \times \mathbb{C})$ and $u$ is the map given by the deformation parameter.
\end{definition}

From now on, by essential smoothing $\tilde{X}$ of $X$, we refer to a member of the family described in Definition \ref{essential smoothing}.

A generic deformation $\tilde{F}$ of the map $F$ defines an essential smoothing of the EIDS $(X, 0)$ (according to the Thom transversality theorem). An essential smoothing is, in general, not smooth (for $N \ge (m-t+2)(n-t+2)$). Its singular locus is $\tilde{F}^{-1}(M_{m,n}^{t-1})$; the singular locus of the latter is $\tilde{F}^{-1}(M_{m,n}^{t-2})$, etc. The representation of $\tilde{X}$ as the union
\[
\tilde{X} = \dot{\bigcup}_{1 \le i \le t} \tilde{F}^{-1}(M_{m,n}^i \setminus M_{m,n}^{i-1})
\]
is a Whitney stratification of it. An essential smoothing of an EIDS $(X, 0)$ of type $(m, n, t)$ is a \textit{genuine smoothing} if and only if $N < (m-t+2)(n-t+2)$.

\subsection{Poincaré-Hopf indices of 1-forms on EIDS}
 In \cite{eg-determinantal}, the authors introduce the Poincaré-Hopf-Nash index of a $1$-form as an extension of the Poincaré-Hopf index defined for manifolds and isolated complete intersections to the context of determinantal singularities. We recall its definition.

Let $(X, 0) = F^{-1}(M_{m,n}^t) \subset (\mathbb{C}^N, 0)$ be an EIDS of type $(m, n, t)$, and let $\omega$ be a germ of a (complex) $1$-form on $(\mathbb{C}^N, 0)$ whose restriction to $(X, 0)$ has an isolated singular point at the origin. This means that the restrictions of the $1$-form $\omega$ to each stratum $X_i \setminus X_{i-1}$, for $i \le t$, have no singular points in a punctured neighborhood of the origin. Let $\pi: \hat{M}_{m,n}^t \to M_{m,n}^t$ be the Nash transform of the variety $M_{m,n}^t \subset \text{Mat}(m, n; \mathbb{C})$. We observe that $\hat{M}_{m,n}^t$ is a resolution of the variety $M_{m,n}^t$ (see \cite{eg-determinantal}).

 Let $\hat{M}_{m,n}^t \times_{M_{m,n}^t} \tilde{X}$ be the fibre product of the spaces $\hat{M}_{m,n}^t$ and $\tilde{X}$ over the variety $M_{m,n}^t$.

The following diagram illustrates the full setting we will work on.

\[
\begin{tikzcd}
\hat{\mathcal{T}} \arrow[d] & & \\
\hat{M}_{m,n}^t \times_{M_{m,n}^t} \tilde{X} \arrow[d,"\Pi"'] \arrow[r] & \hat{M}_{m,n}^t \arrow[d, "\pi"] & \\
\tilde{X} \arrow[d, "j"'] \arrow[r, "\tilde{F}"] & M_{m,n}^t & \\
\mathbb{C}^N & X \arrow[l, hook', "j"] \arrow[u, "F"'] & \overline{X} \arrow[l,"\nu"] & \overline{\mathcal{T}} \arrow[l]
\end{tikzcd}
\]

\vspace{0,5cm}

The (smooth) complex analytic variety $\hat{M}_{m,n}^t \times_{M_{m,n}^t} \tilde{X}$ has two important properties: it is a resolution and it is the Nash transform of the variety $\tilde{X}$ (see \cite{eg-determinantal}). 
Let $j:\tilde{X} \hookrightarrow U \subset \mathbb{C}^N$ be the inclusion map. The lifting $\omega := (j \circ \Pi)^* \omega$ of the $1$-form $\omega$ is a $1$-form on a (nonsingular) complex analytic manifold $\overline{X}$ without singular points outside the preimage of a small neighborhood of the origin. 
Hence the $1$-form $\omega$ defines a nonvanishing section $\hat{\omega}$ of the dual bundle $\hat{\mathcal{T}}^*$ over the preimage of the intersection $\tilde{X} \cap S_\eps$ of the variety $\tilde{X}$ with the sphere $S_\eps$.

\begin{definition}\label{PHN-index definition}
The \textit{Poincaré-Hopf-Nash index} (PHN-index) $\mathrm{ind}_{PHN}(\omega, X, 0)$ of the $1$-form $\omega$ on the EIDS $(X, 0)$ is the obstruction to extending the nonzero section $\hat{\omega}$ of the dual Nash bundle $\hat{\mathcal{T}}^*$ from the preimage of the boundary $S_\eps = \partial B_\eps$ of the ball $B_\eps$ to the preimage of its interior, i.e., to the manifold $\hat{M}_{m,n}^t \times_{M_{m,n}^t} \tilde{X}$. 

\end{definition}

A geometric interpretation for the PHN-index is as follows.

\begin{proposition}\cite{eg-determinantal}\label{caracterization PHN}
The PHN-index $\mathrm{ind}_{PHN}(\omega, X, 0)$ of the $1$-form $\omega$ on the EIDS $(X, 0)$ is equal to the number of nondegenerate singular points of a generic deformation $\tilde{\omega}$ of the $1$-form $\omega$ on the nonsingular part $\tilde{X}_{reg} = \tilde{F}^{-1}(M_{m,n}^t \setminus M_{m,n}^{t-1})$ of an essential smoothing $\tilde{X}$ of the singularity $(X, 0)$.
\end{proposition}


In \cite{eg-determinantal}, the authors provided formulas to relate the radial index and the PHN-index of a $1$-form.
Let $\chi(X, 0) := \chi(\tilde{X} \cap B_\eps)$ for an essential smoothing $\tilde{X}$ of the singularity $(X, 0)$ and $B_{\epsilon}$ a closed ball around the origin. Let us recall that $\tilde{X}_i = \tilde{F}^{-1}(M_{m,n}^i)$, for $i=1, \dots, t$, is an essential smoothing of the EIDS $(X_i, 0)$ (of type $(m, n, i)$).
In the context of determinantal singularities, the normal index associated with the stratification given by the rank has combinatorial formulas. For $i\leq j,$ one has (see \cite{eg-determinantal}) \begin{center}
   $n_{ij}=(-1)^{(m+n)(j-i)}\binom{m - i}{m - j}$ and $m_{ij}=(-1)^{(m+n+1)(j-i)}\binom{m-i}{m-j}$.
\end{center}
 
\begin{proposition}\label{prop4det}\cite[Proposition 5]{eg-determinantal}
The following equation holds:
 \begin{center}
      $ind_{PHN}(\omega;X,0)
    =\displaystyle\sum_{i=1}^{t}m_{it}\big(ind_{rad}(\omega;X_i,0)+(-1)^{\dim X_i}\overline{\chi}(X_i,0)\big).$
 \end{center}
\end{proposition}

\section{The local Euler obstruction of a $1$-form on a determinantal variety}\label{Section Euler obstruction}

Let $(X, 0) = F^{-1}(M_{m,n}^t) \subset (\mathbb{C}^N, 0)$ be an EIDS of type $(m, n, t)$, and let $\omega$ be a germ of a (complex) $1$-form on $(\mathbb{C}^N, 0)$ whose restriction to $(X, 0)$ has an isolated singular point at the origin.

We aim to compare the local Euler obstruction of $\omega$ on $(X,0)$ with the Poincaré-Hopf-Nash index on $(X,0)$. 

Let $\nu : \hat{X} \to X$ be the Nash transformation of the variety $X$ and $\hat{T}$ be the Nash bundle over $\hat{X}$ as presented in Section \ref{Euler onstruction of a 1-Form}. The $1$-form $\omega$ gives rise to a section $\hat{\omega}$ of the dual Nash bundle $\hat{T}^*$ over the Nash transform $\hat{X}$ without singular points outside the preimage of the origin.

This leads to the following consequence of Definitions \ref{Eu definition} and \ref{PHN-index definition}.
\begin{corollary}
   The PHN-index of the $1$-form $\omega$ on the EIDS $(X, 0)$ is equal to the Euler obstruction of the $1$-form $\omega$ on the essential smoothing $\tilde{X}$ of $X$ at the origin, more precisely, 
    \[
    Eu_{\tilde{X},0}\omega=\mathrm{ind}_{PHN}(\omega, X, 0).
    \]
\end{corollary}

\begin{lemma}\label{Lemma Eu}
Let $(X, 0)$ be an EIDS of type $(m,n,t)$ and $\omega$ be a germ of a (complex) $1$-form on $(\mathbb{C}^N, 0)$ whose restriction to $(X,0))$  has an isolated singular point at the origin. Then

    \begin{align*}
\mathrm{Eu}_{X,0} \omega &= (-1)^{(m+n+1)t}\binom{m}{m-t} + \mathrm{ind}_{\mathrm{PHN}}(\omega, X, 0) + \sum_{i=1}^t (-1)^{\dim X_i - 1} m_{it}\overline{\chi}(X_i, 0),
\end{align*}
\noindent where $\chi(X_i, 0) := \chi(\tilde{X}_i \cap B_\eps)$, for $i=1,\ldots,t$, $B_{\epsilon}$ is a closed ball around the origin and $m_{it}=(-1)^{(m+n+1)(t-i)}\binom{m-i}{m-t}.$.
\end{lemma}
\begin{proof}
      By Corollary \ref{Euradial},
$$ \mathrm{Eu}_{X,0} \omega = \sum_{i=0}^{t} m_{it} \cdot \mathrm{ind}_{rad}(\omega;\overline{V}_{i},0) $$
where $V_i$ are strata of a Whitney stratification of $(X,0)$. Let $X_i=F^{-1}(M_{m,n}^i)$. Since $(X,0)$ is an EIDS, $ V_i =X_i \setminus X_{i-1}$, hence

$$ \mathrm{Eu}_{X,0} \omega = \sum_{i=0}^{t} m_{it} \cdot \mathrm{ind}_{rad}(\omega;\overline{X_i \setminus X_{i-1}},0) = \sum_{i=0}^{t} m_{it} \cdot\mathrm{ind}_{rad}(\omega;X_{i},0). $$

Applying Proposition \ref{prop4det} for each determinantal singularity $(X_i,0)$, one has

\begin{eqnarray*}
m_{1t} \mathrm{ind}_{rad}(\omega;X_{1},0) &=& m_{1t} \big(n_{11} \mathrm{ind}_{\mathrm{PHN}}(\omega, X_1, 0)  + (-1)^{\dim X_1 - 1} \overline{\chi}(X_1, 0)\big) \\
m_{2t} \mathrm{ind}_{rad}(\omega;X_{2},0) &=& m_{2t} \big( n_{12} \mathrm{ind}_{\mathrm{PHN}}(\omega, X_1, 0) + n_{22} \mathrm{ind}_{\mathrm{PHN}}(\omega, X_2, 0)\\  &+& (-1)^{\dim X_2 - 1} \overline{\chi}(X_2, 0)\big) \\
&\vdots& \\
m_{t-1, t} \mathrm{ind}_{rad}(\omega;X_{t-1},0)&=& m_{t-1, t} \big( n_{1, t-1} \mathrm{ind}_{\mathrm{PHN}}(\omega, X_1, 0) + \dots + n_{t-1, t-1} \mathrm{ind}_{\mathrm{PHN}}(\omega, X_{t-1}, 0)  \\
&& +\ (-1)^{\dim X_{t-1} - 1} \overline{\chi}(X_{t-1}, 0)\big) \\
m_{tt} \mathrm{ind}_{rad}(\omega;X_{t},0) &=& m_{tt} \big( n_{1t} \mathrm{ind}_{\mathrm{PHN}}(\omega, X_1, 0) + \dots + n_{tt} \mathrm{ind}_{\mathrm{PHN}}(\omega, X_t, 0)\big)\\
&& +\ (-1)^{\dim X_t - 1} \overline{\chi}(X_t, 0)\big)
\end{eqnarray*}

Since $X_0=\{0\}, \mathrm{ind}_{rad}(\omega;X_{0},0)=1$. Thus 

\begin{eqnarray*}
\mathrm{Eu}_{X,0} \omega &=& m_{0t} + m_{1t} \mathrm{ind}_{X_1, 0} \omega + \dots + m_{t-1, t} \mathrm{ind}_{X_{t-1}, 0} \omega + m_{tt} \mathrm{ind}_{X_t, 0} \omega \\
&=& m_{0t} + (n_{11} m_{1t} + n_{12} m_{2t} + \dots + n_{1, t-1} m_{t-1, t} + n_{1t} m_{tt}) \mathrm{ind}_{\mathrm{PHN}}(\omega, X_1, 0) \\
&+& (n_{22} m_{2t} + \dots + n_{2, t-1} m_{t-1, t} + n_{2t} m_{tt}) \mathrm{ind}_{\mathrm{PHN}}(\omega, X_2, 0)\\
&\vdots& \\
&+& (n_{t-1, t-1} m_{t-1, t} + n_{t-1, t} m_{tt}) \mathrm{ind}_{\mathrm{PHN}}(\omega, X_{t-1}, 0) + n_{tt} m_{tt} \mathrm{ind}_{\mathrm{PHN}}(\omega, X_t, 0) \\
&+& (-1)^{\dim X_1 - 1} m_{1t}\overline{\chi}(X_1, 0) + \dots + (-1)^{\dim X_t - 1} m_{tt}\overline{\chi}(X_t, 0)
\end{eqnarray*}

Since $\sum_{i \le j \le k} n_{ij} m_{jk} = \delta_{ik}$,
\begin{align*}
\mathrm{Eu}_{X,0} \omega &= m_{0t} + \mathrm{ind}_{\mathrm{PHN}}(\omega, X_t, 0) + \sum_{i=1}^t (-1)^{\dim X_i - 1} m_{it}\overline{\chi}(X_i, 0) \\
&= m_{0t} + \mathrm{ind}_{\mathrm{PHN}}(\omega, X, 0) + \sum_{i=1}^t (-1)^{\dim X_i - 1} m_{it}\overline{\chi}(X_i, 0) \\
&= (-1)^{(m+n+1)t}\binom{m}{m-t} + \mathrm{ind}_{\mathrm{PHN}}(\omega, X, 0) + \sum_{i=1}^t (-1)^{\dim X_i - 1} m_{it}\overline{\chi}(X_i, 0).
\end{align*}
\end{proof}

Based on this lemma, we can split off the structural data that is independent of $\omega$ in order to compute the Euler obstruction of the $1$-form $\omega.$

\begin{theorem}\label{main theorem}
If $l:X\to\mathbb{C}$ is a generic linear function, then
 \begin{equation}\label{Formula teorema 1}
\mathrm{Eu}_{X,0} \omega 
= \mathrm{ind}_{\mathrm{PHN}}(\omega, X, 0)-\mathrm{ind}_{\mathrm{PHN}}(dl, X, 0).
\end{equation}
   
\end{theorem}
\begin{proof} Applying Lemma \ref{Lemma Eu} for any linear function $l:X\to\mathbb{C}$,

\begin{align*}
\mathrm{Eu}_{X,0} \mathrm{dl} &= m_{0t} + \mathrm{ind}_{\mathrm{PHN}}(\mathrm{dl}, X, 0) + \sum_{i=1}^t (-1)^{\dim X_i - 1} m_{it}\overline{\chi}(X_i, 0).
\end{align*}

Therefore, \begin{align*}
\mathrm{Eu}_{X,0} \omega - \mathrm{Eu}_{X,0} \mathrm{dl}
&= \mathrm{ind}_{\mathrm{PHN}}(\omega, X, 0)-\mathrm{ind}_{\mathrm{PHN}}(dl, X, 0).
\end{align*}

If $l$ is sufficiently generic, then $\mathrm{Eu}_{X,0} \mathrm{dl}=0,$ (see \cite[Proposition 2.4]{BMPS}), which concludes the proof.
\end{proof}

The geometric characterization for the PHN-index provided by Proposition \ref{caracterization PHN} together with Theorem \ref{main theorem} provides an algebraic formula to compute the Euler obstruction of a $1$-form on an IDS $(X,0)$.

\begin{example}\label{example 1}
    Let $F(x_1,\dots,x_6) = \begin{pmatrix} x_1 & x_2 & x_3 \\ x_4 & x_5 & x_6 \end{pmatrix} $ and consider the $4$-dimensional IDS  $X = F^{-1}(M_{2,3}^2)$ which is known to be rigid. Consider the $1$-form $$\omega(x_1,\dots,x_6)=2x_1dx_1+3x_2^2x_3dx_2+x_2^3dx_3-x_5dx_4-x_4dx_5$$ and a sufficiently general linear $1$-form, say e.g. $$l(x_1,\ldots,x_6)=2dx_1 + 5dx_2 + 3dx_3 - 5dx_4 - 7dx_5 + 11dx_6.$$
Denote by $I_{\tilde{X}}$ the defining ideal of $\tilde{X}$, and let $I(\tilde{X},\tilde{\omega})$ be the ideal generated by the $3\times3$-minors of $J(\tilde{X},\tilde{\omega})$, where $J(\tilde{X},\tilde{\omega})$ is the Jacobian matrix of the defining functions of $(\tilde{X},0)$ augmented by the coefficients of a generic deformation $\tilde{\omega}$ of $\omega$, namely 

$$J(\tilde{X},\tilde{\omega})=\begin{bmatrix}
x_5 & -x_4 & 0    & -x_2 & x_1  & 0   \\
x_6 & 0    & -x_4 & -x_3 & 0    & x_1 \\
0   & x_6  & -x_5 & 0    & -x_3 & x_2 \\
2x_1+2 & 3x_2^2x_3+5  & x_2^3+3  & -x_5-5  & -x_4-7  & 11
\end{bmatrix}.$$

Since $(X,0)$ is a rigid singularity, an essential smoothing of $X$ is $X$ itself, and $\tilde{X}$ has isolated singularity at the origin. On the other hand, a generic deformation $\tilde{\omega}=\omega+tl,t\in\mathbb{C}$ has only non-degenerate (and therefore isolated) singular points in $\tilde{X}_{reg}=X\setminus\{0\}$. Denoting by $\Sigma\tilde{\omega}$ the set of singular points of $\tilde{\omega},$ we have $\Sigma\tilde{\omega}|_{\tilde{X}}=V(I_{\tilde{X}}+I(\tilde{X},\tilde{\omega}))=\{0\}\cup\Sigma\tilde{\omega}|_{\tilde{X}_{reg}}.$

By the Principle of Conservation of the Number, $$\mathrm{ind}_{\mathrm{PHN}}(\omega, X, 0)=\sum_{P_i\in\Sigma\tilde{\omega}} \mathrm{ind}_{\mathrm{PHN}}(\tilde{\omega}, X, P_i)$$

In order to count (with multiplicity) the number of non-degenerate singular points of $\tilde{\omega}$ in $X\setminus\{0\}$, one may compute the dimension of vector spaces $\dim_{\mathbb{C}}\Big({\mathcal{O}_{6}}\Big/{I_{\tilde{X}}+I(\tilde{X},\tilde{\omega})}\Big)$ removing the non-degenerate singular points that arise from the origin, which can be computed by the dimension $\dim_{\mathbb{C}}\Big({\mathcal{O}_{6,0}}\Big/{I_{\tilde{X}}+I(\tilde{X},\tilde{\omega})}\Big)$.

These dimensions are easily computable using the function \verb|vector_space_dim| in OSCAR \cite{OSCAR-book,OSCAR}, therefore:

$$\mathrm{ind}_{\mathrm{PHN}}(\mathrm{dl}, X, 0)=\dim_{\mathbb{C}}\Big({\mathcal{O}_{6}}\Big/{I_{\tilde{X}}+I(\tilde{X},\mathrm{dl})}\Big) - \dim_{\mathbb{C}}\Big({\mathcal{O}_{6,0}}\Big/{I_{\tilde{X}}+I(\tilde{X},\mathrm{dl})}\Big) = 7 - 7 =0$$ 
and

 $$\mathrm{ind}_{\mathrm{PHN}}(\omega, X, 0)=\dim_{\mathbb{C}}\Big({\mathcal{O}_{6}}\Big/{I_{\tilde{X}}+I(\tilde{X},\tilde{\omega})}\Big) - \dim_{\mathbb{C}}\Big({\mathcal{O}_{6,0}}\Big/{I_{\tilde{X}}+I(\tilde{X},\tilde{\omega})}\Big)  = 20 - 7 = 13.$$

Hence, 
\begin{equation*}
\mathrm{Eu}_{X,0} \omega = \mathrm{ind}_{\mathrm{PHN}}(\omega, X, 0) - \mathrm{ind}_{\mathrm{PHN}}(\mathrm{dl}, X, 0) = 13 - 0 = 13.
\end{equation*}
\end{example}

In \cite{gaffney2021equisingularity}, the authors defined a polar multiplicity $m_d(X,0)$ associated with an EIDS $(X,0)$ of dimension $d$ as the multiplicity of the relative polar curve of an essential smoothing of $X$ at the origin. We have the following consequence.

\begin{corollary}\label{PHN and m_d}
    Let $\omega$ be a $1$-form defined on an EIDS $(X,0)$ of dimension $d$. Then \begin{align*}
\mathrm{Eu}_{X,0} \omega
&= \mathrm{ind}_{\mathrm{PHN}}(\omega, X, 0)- m_d(X,0).
\end{align*}
\end{corollary}
\begin{proof}
The $d$-polar multiplicity $m_d(X,0)$ is equal to the Poincaré-Hopf-Nash index $\mathrm{ind}_{\mathrm{PHN}}(\mathrm{dl}, X, 0)$, where $l$ is  a generic linear function on $X$, see \cite{gaffney2021equisingularity}. Hence, by Theorem \ref{main theorem}\begin{align*}
\mathrm{Eu}_{X,0} \omega
&= \mathrm{ind}_{\mathrm{PHN}}(\omega, X, 0)- m_d(X,0).
\end{align*}
\end{proof}

The Euler obstruction of a function was introduced in \cite{BMPS} as a generalization of the Milnor number for a holomorphic function $f:X\to\C$ with a stratified isolated singularity at the origin defined on a complex analytic variety. The local Euler obstruction of $\mathrm{df}$ at the origin is closely related to the Euler obstruction of $f$ at the origin. More precisely,  $\mathrm{Eu}_{X,0} \mathrm{df}=(-1)^{\dim X}\mathrm{Eu}_{f,X} (0)$ (see \cite{ebeling2005radial}). Therefore, one has the following consequence of Theorem \ref{main theorem}.

\begin{theorem}\label{teorema2}
Let $f:X\to\mathbb{C}$ be a function with a stratified isolated singularity at the origin and $l:X\to\mathbb{C}$ be a generic linear function. Then
    \begin{align*}
\mathrm{Eu}_{f,X} (0)
&= (-1)^{\dim X}\big(\mathrm{ind}_{\mathrm{PHN}}(\mathrm{df}, X, 0)-\mathrm{ind}_{\mathrm{PHN}}(\mathrm{dl}, X, 0)\big).
\end{align*}
\end{theorem}

\begin{example}
    Consider the IDS  $X = F^{-1}(M_{2,4}^2)$, for $$F(x_1,\dots,x_8) = \begin{pmatrix} x_1 & x_2 & x_3 & x_4 \\ x_5 &x_6 & x_7 & x_8\end{pmatrix}. $$ 
    
    Consider the function germ $f(x_1,\dots,x_8)=x_1^2+3x_2^2x_3+x_6^3-x_7^4$ and the generic linear function $l(x_1,\ldots,x_8)=2x_1 + 5x_2 + 3x_3 - 5x_4 - 7x_5 + 11x_6 - 8x_7+3x_8.$ Since $(X,0)$ is a rigid singularity, an essential smoothing of $X$ is $X$ itself, and $\tilde{X}$ has isolated singularity at the origin. 

Following the arguments of Example \ref{example 1} and using the function \verb|vector_space_dim| as before, 

$$\mathrm{ind}_{\mathrm{PHN}}(\mathrm{dl}, X, 0)=\dim_{\mathbb{C}}\Big({\mathcal{O}_{8}}\Big/{I_{\tilde{X}}+I(\tilde{X},\mathrm{dl})}\Big) - \dim_{\mathbb{C}}\Big({\mathcal{O}_{8,0}}\Big/{I_{\tilde{X}}+I(\tilde{X},\mathrm{dl})}\Big) = 39 - 39 = 0$$ 
and $$\mathrm{ind}_{\mathrm{PHN}}(\omega, X, 0)=\dim_{\mathbb{C}}\Big({\mathcal{O}_{8}}\Big/{I_{\tilde{X}}+I(\tilde{X},\omega)}\Big) - \dim_{\mathbb{C}}\Big({\mathcal{O}_{8,0}}\Big/{I_{\tilde{X}}+I(\tilde{X},\omega)}\Big) = 45-39 = 6.$$

Therefore, 
\begin{equation*}
Eu_{f,X} (0) = (-1)^{5}(\mathrm{ind}_{\mathrm{PHN}}(df, X, 0) - \mathrm{ind}_{\mathrm{PHN}}(\mathrm{dl}, X, 0)) = -( 6 - 0) = -6.
\end{equation*}
\end{example}

\section{PHN-index for smoothable isolated determinantal singularities}

In this section, we verify that the framework developed in Section \ref{Section Euler obstruction} is consistent with previously established results in the literature in the absence of rigid singularities.

Let $(X, 0)$ be a germ of an IDS defined by $F^{-1}(M_{m,n}^s)$. For a matrix $A \in M_{m,n}$, denote by $\tilde{F}_A : (\mathbb{C}^N, 0) \to M_{m,n}$ the map $F_A(x) = F(x) + tA$ and set $X_A := F_A^{-1}(M_{m,n}^s)$, where $t\in\mathbb{C}.$. Suppose that $s=1$ or $N<(m-s+1)(n-s+1).$ In \cite{nuno2013vanishing}, the authors proved that for $0<|t|\ll1,$ the variety $X_A$ is smooth and its Euler characteristic $\chi(X_A)$ is independent of $t,$ motivating the definition of the vanishing Euler characteristic of $(X,0)$ 
$$v(X, 0) := (-1)^{\dim X}(\chi(X_A) - 1),$$
where $\operatorname{rank}(F_A(x)) = s - 1$ for all $x \in X_A$. In \cite{FKZach2}, the authors obtained explicit formulas for $\nu(X,0)$ using the Tjurina transform for Cohen-Macaulay codimension $2$ smoothable determinantal singularities, providing algorithmic methods to compute it by relating it to a suitable complete intersection singularity of the same topological type via Tjurina transformation. Both approaches show that the vanishing Euler characteristic generalizes the Milnor number from the hypersurface case \cite{milnorsingular}.

But this is not the only possible generalization of the Milnor number of complete intersections to the context of smoothable determinantal singularities. 
Given $f : (X, 0) \to (\mathbb{C}, 0)$ a function-germ with an isolated singularity, the \textit{determinantal Milnor number} of $f$ is defined in \cite{nuno2013vanishing} by
$$\mu_D(f) := \#\Sigma\left(f_a|_{X_A}\right),$$

\noindent with $A \in M_{m,n}$ and $a = (a_1, \ldots, a_N) \in \mathbb{C}^N$ generic such that $f_a|_{U_A} : U_A \to \mathbb{C}$ is a Morse function, where $f_a(x_1, \ldots, x_N) = f(x_1, \ldots, x_N) + a_1x_1 + \cdots + a_Nx_N$ and $U_A\subset X_A$ is an open subset of $X_A$ . In that paper, the authors proved that the vanishing Euler characteristic and the determinantal Milnor number do not coincide; however, they are related (see \cite[Theorem 5.4]{nuno2013vanishing}).

Our results from Section \ref{Section Euler obstruction} generalize the latter generalization of the Milnor number:

\begin{proposition}\label{PHN e mud}
    Let $(X,0)$ be a smoothable isolated determinantal singularity, $f:X\to\mathbb{C}$ be a function with an isolated singularity at the origin. Then $\mathrm{ind}_{\mathrm{PHN}}(\mathrm{df}, X, 0) =\mu_D(f).$ 
\end{proposition}

\begin{proof}
    By Proposition \ref{caracterization PHN}, $ \mathrm{ind}_{\mathrm{PHN}}(\mathrm{df}, X, 0)$ is the number of nondegenerate singular points of a generic deformation $\tilde{\mathrm{df}}$ of the $1$-form $\mathrm{df}$ on the nonsingular part $\tilde{X}_{reg}$ of an essential smoothing $\tilde{X}$ of the singularity $(X, 0)$. Any generic deformation $\mathrm{f+t\cdot l} $ of $\mathrm{f}$ (see \cite[Theorem 3.2]{Ms1}) gives a generic deformation $\mathrm{d(f+t\cdot l)=df+t\cdot dl}$ of $\mathrm{df}$, where $|t|<<1$ and $l$ is a generic linear function. In particular,  $f_a(x_1,\ldots,x_n)=f(x_1,\ldots,x_n)+a_1x_1+\cdots+a_nx_n$
    provides a generic deformation $\mathrm{df_a}$ of $\mathrm{df}.$
    Since $(X,0)$ is smoothable, $\tilde{X}_{reg}=\tilde{X}$ and, by \cite[Lemma 5.1]{nuno2013vanishing}, one may take $\tilde{X}=X_A$. Hence,
    
    \begin{eqnarray*}
    \mathrm{ind}_{\mathrm{PHN}}(\mathrm{df}, X, 0) &=& \# \ \Sigma\tilde{df}|_{\tilde{X}_{reg}}=\# \ \Sigma\tilde{df_a}|_{X_A}\\ &=& \# \ \Sigma\tilde{f_a}|_{X_A}=\mu_D(f).
    \end{eqnarray*}
\end{proof}

In \cite{ament2016euler}, the authors provided a relation between the Euler obstruction of a function with an isolated singularity defined on a smoothable IDS using the determinantal Milnor number and the $d$-th polar multiplicity. More precisely,  they proved that \begin{center}
   $ \mathrm{Eu}_{f,X} (0)= (-1)^{\dim X}\big(\mu_D(f)-m_d(X,0)\big)$.
\end{center}

This formula can be viewed as the restriction of the formula developed in Theorem \ref{teorema2} to the smoothable case, because by Proposition \ref{PHN e mud}, $\mu_D(f)= \mathrm{ind}_{\mathrm{PHN}}(\mathrm{df}, X, 0)$, and by the same argument as in the proof of Corollary \ref{PHN and m_d}, $m_d(X,0)= \mathrm{ind}_{\mathrm{PHN}}(\mathrm{dl}, X, 0).$

In the following example, we apply all this to study a smoothable determinantal singularity of codimension 3, which is known to allow two different matrix structures with corresponding determinantal smoothings. This illustrates that we are using the respective (different) matrices and their perturbations, while the resulting invariant is independent of the chosen matrix type.

\begin{example}
    Consider the IDS  $X = F^{-1}(M_{2,4}^2)$, for $$F(x_1,\dots,x_5) = \begin{pmatrix} x_1 & x_2 & x_3 & x_4 \\ x_2 & x_3 & x_4 & x_5 \end{pmatrix}, $$ that is,  $$X=V(x_1x_3-x_2^2,x_1x_4-x_2x_3,x_1x_5-x_2x_4,x_2x_4-x_3^2,x_2x_5-x_3x_4,x_3x_5-x_4^2).$$
    
    Consider the function germ $f(x_1,\dots,x_5)=x_1^3+x_5+x_2x_3$ and the generic linear function $l(x_1,\ldots,x_5)=2x_1 + 5x_2 + 3x_3 - 5x_4 - 7x_5 .$ An (essential) smoothing of $X$ is $\tilde{X}=\tilde{F}^{-1}(M^2_{2,4})$, for  $$\tilde{F}(x_1,\dots,x_5) = \begin{pmatrix}x_1 & x_2 & x_3 & x_4 \\ x_2 & x_3 & x_4-\frac{1}{100} & x_5 \end{pmatrix}. $$

Using the function \verb|vector_space_dim| as before,

$$m_d(X,0)=\mathrm{ind}_{\mathrm{PHN}}(\mathrm{dl}, X, 0)=\dim_{\mathbb{C}}\Big({\mathcal{O}_{5}}\Big/{I_{\tilde{X}}+I(\tilde{X},\mathrm{dl})}\Big)  = 4$$ 
and $$\mu_D(f)=\mathrm{ind}_{\mathrm{PHN}}(\mathrm{df}, X, 0)=\dim_{\mathbb{C}}\Big({\mathcal{O}_{5}}\Big/{I_{\tilde{X}}+I(\tilde{X},\mathrm{df})}\Big) = 12.$$

Therefore, 
\begin{equation*}
Eu_{f,X} (0) = (-1)^{3}(\mathrm{ind}_{\mathrm{PHN}}(\mathrm{df}, X, 0) - \mathrm{ind}_{\mathrm{PHN}}(\mathrm{dl}, X, 0)) = (-1)^3(12 - 4) = -8.
\end{equation*}

Another representation of $X$ is $X = G^{-1}(M_{2,3}^2)$, for $$G(x_1,\dots,x_5) = \begin{pmatrix}
x_1 & x_2 & x_3 \\
x_2 & x_3 & x_4 \\
x_3 & x_4 & x_5
\end{pmatrix}. $$ 
    
It is easy to see that the ideals generated by the $2\times2$- minors of $G$ and of $F$ are equal. An (essential) smoothing of $X$ is $\tilde{X}=\tilde{G}^{-1}(M^2_{2,3})$, for  $$\tilde{G}(x_1,\dots,x_5) = \begin{pmatrix}
x_1 & x_2 & x_3 \\
x_2 & x_3-\frac{1}{4} & x_4 \\
x_3 & x_4 & x_5
\end{pmatrix}. $$

Using the function \verb|vector_space_dim| as before,

$$m_d(X,0)=\mathrm{ind}_{\mathrm{PHN}}(\mathrm{dl}, X, 0)=\dim_{\mathbb{C}}\Big({\mathcal{O}_{5}}\Big/{I_{\tilde{X}}+I(\tilde{X},\mathrm{dl})}\Big)  = 3$$ 
and $$\mu_D(f)=\mathrm{ind}_{\mathrm{PHN}}(\mathrm{df}, X, 0)=\dim_{\mathbb{C}}\Big({\mathcal{O}_{5}}\Big/{I_{\tilde{X}}+I(\tilde{X},\mathrm{df})}\Big) = 11.$$

Therefore, 
\begin{equation*}
Eu_{f,X} (0) = (-1)^{3}(\mathrm{ind}_{\mathrm{PHN}}(\mathrm{df}, X, 0) - \mathrm{ind}_{\mathrm{PHN}}(\mathrm{dl}, X, 0)) = (-1)^3(11 - 3) = -8.
\end{equation*}
\end{example}

We conclude by noting that the PHN‑index of a $1$-form serves as a generalization of the Milnor number in the presence of rigid singularities, in a manner consistent with previous extensions of the Milnor number in settings without rigidity.

\section*{Acknowledgements}

The authors are grateful to Matthias Zach and Marcel Salmon for their fruitful discussions and remarks.

\bibliography{references}
\bibliographystyle{amsplain}

\end{document}